\documentclass[11pt,a4paper,english,reqno]{amsart}
\usepackage{amsmath,amssymb,amsfonts,epsfig,mathrsfs}
\usepackage[T1]{fontenc}

\usepackage{color}
\usepackage{array}
\usepackage{amsthm}
\usepackage{amstext}
\usepackage{graphicx}
\usepackage{setspace}
\usepackage[margin=2.5cm]{geometry}
\usepackage{bbm}
\usepackage{color}
\usepackage{enumitem}
\usepackage{undertilde}
\usepackage{amscd,psfrag}
\usepackage{yhmath}
\usepackage[mathscr]{eucal}

\usepackage{slashed}

\makeatletter
\pdfpageheight\paperheight
\pdfpagewidth\paperwidth

\allowdisplaybreaks[4]

\usepackage{pgfplots}

\usepackage{epstopdf}
\usepackage{chngcntr}
\counterwithin{figure}{section}
\usepackage{mathrsfs}

\usepackage{indentfirst}	

\usepackage[normalem]{ulem}
\theoremstyle{plain}

\graphicspath{C:/Users/Administrator/Documents/Maths/My Writings/[new] isom emb of negative K surfaces}

\newtheorem{definition}{Definition}[section]
\newtheorem{theorem}[definition]{Theorem}
\newtheorem*{theorem*}{Theorem}

\newtheorem*{remark*}{Remark}
\newtheorem*{sideremark*}{Side Remark}\newtheorem*{mt*}{Main Theorem}

\newtheorem*{claim*}{Claim}
\newtheorem*{q*}{Question}
\newtheorem{lemma}[definition]{Lemma}

\newtheorem*{corollary*}{Corollary}
\newtheorem*{proposition*}{Proposition}

\newtheorem{proposition}[definition]{Proposition}

\newcommand{\R}{\mathbb{R}}

\newcommand{\na}{\nabla}

\newcommand{\dd}{{\rm d}}
\newcommand{\p}{\partial}
\newcommand{\e}{\epsilon}

\newcommand{\map}{\rightarrow}
\newcommand{\G}{\Gamma}
\newcommand{\geu}{{g_{\rm Eucl}}}
\newcommand{\two}{{\rm II}}
\newcommand{\barna}{{\overline{\na}}}
\newcommand{\E}{\mathcal{E}}
\newcommand{\M}{{\mathcal{M}}}
\newcommand{\tl}{\widetilde{L}}
\newcommand{\tn}{\widetilde{N}}
\newcommand{\tm}{\widetilde{M}}
\newcommand{\tg}{\widetilde{\G}}
\newcommand{\sone}{{\mathcal{S}^{(1)}}}
\newcommand{\stwo}{{\mathcal{S}^{(2)}}}
\newcommand{\source}{{\mathcal{S}}}
\newcommand{\gtau}{{g^{(\tau)}}}
\newcommand{\TT}{\Theta}
\newcommand{\gaa}{g^{(\alpha)}}
\newcommand{\FF}{\mathcal{F}}
\newcommand{\GG}{\mathcal{G}}

\newcommand{\pone}{{\wp^{(1)}}}
\newcommand{\ptwo}{{\wp^{(2)}}}

\numberwithin{equation}{section}
\numberwithin{figure}{section}

\title{On the existence of $C^{1,1}$ isometric immersions of several classes of negatively curved surfaces into $\R^3$}
\author{Siran Li}
\address{Siran Li: Department of Mathematics, Rice University, MS 136
P.O. Box 1892, Houston, Texas, 77251-1892, USA; \, $\bullet$ \,  Department of Mathematics, McGill University, Burnside Hall, 805 Sherbrooke Street West, Montreal, Quebec, H3A 0B9, Canada.}

\email{\texttt{Siran.Li@rice.edu}}

\date{\today}

\begin{document}

\maketitle

\begin{abstract}
We prove the existence of $C^{1,1}$ isometric immersions of several classes of metrics on surfaces $(\M,g)$ into the three-dimensional Euclidean space $\R^3$, where the metrics $g$ have strictly negative curvature. These include the standard hyperbolic plane,  generalised helicoid-type metrics and generalised Enneper metrics. Our proof is based on the method of compensated compactness and invariant regions in hyperbolic conservation laws, together with several observations on the geometric quantities (Gauss curvature, metric components etc.) of negatively curved surfaces. 

\end{abstract}

\section{Introduction}
	The existence of isometric immersions of a given $n$-dimensional Riemannian manifold $(\M,g)$ into the Euclidean space $\R^N$ is a well-known classical problem in differential geometry and non-linear analysis; \textit{cf.} \cite{cartan, hh, janet, nash, rozendorn, yau} and the references cited therein. Even in the case $(n,N)=(2,3)$, namely the isometric immersions of surfaces into the Euclidean $3$-space, the question of existence remains open in the large. 
	
	From the perspective of Partial Differential Equations (PDEs), several distinct approaches have been developed toward the study of isometric immersions. First, in \cite{nash}, Nash directly tacked the defining equations for isometric immersions: let $g$ be the given Riemannian metric on $\M$; an isometric immersion $f: (\M,g) \map (\R^N, \geu)$ satisfies
	\begin{equation}\label{nash eq}
	\dd f \cdot \dd f =g,
	\end{equation}
where $\geu$ denotes the Euclidean metric on $\R^N$. Eq.\,\eqref{nash eq} forms a first-order non-linear PDE system, which is of no definite type (elliptic, parabolic or hyperbolic). Nash proved the existence of $C^\infty$ isometric immersions for $g \in C^\infty$, provided that the co-dimension $(N-n)$ is sufficiently large. In particular, in the case $n=2$, we need co-dimension up to $15$, even if $\M$ is a compact surface. 

The second approach is via the Darboux equation. Fix any unit vector $\mathbf{e} \in \R^3$ and consider $\phi:=f \cdot \mathbf{e}$; the local existence of an isometric immersion $f: (\M,g) \map (\R^3, \geu)$ is equivalent to the following PDE of Monge--Amp\`ere type, named after Darboux:
\begin{equation}\label{darboux eq}
\det(\na^2 \phi) = \kappa \det(g) (1-|\na \phi|^2),
\end{equation}
provided that $|\na \phi|^2 = g^{ij}\p_i \phi \p_j \phi <1$. Throughout the Einstein summation convention is assumed, and $\kappa$ denotes the Gauss curvature of the surface $(\M,g)$. By establishing the elliptic theory for Monge--Amp\`{e}re equations, Nirenberg \cite{nirenberg} proved the existence of isometric immersions of $(\M,g)$ for $\kappa>0$, $g \in C^4$. Notice that whenever $\kappa>0$, $\M$ is a topological sphere, in view of the Gauss--Bonnet theorem. Guan--Li \cite{guanli} extended Nirenberg's result to the case $\kappa \geq 0$, $g \in C^4$, by passing to the limits of the degenerate-elliptic equation \eqref{darboux eq}. The sign of $\kappa$ is crucial to the study of the existence of classical solution to Eq.\,\eqref{darboux eq}, hence,  of the existence of isometric immersions. For example, it is well-known that the pseudo-2-sphere (space form of constant Gauss curvature $-1$) cannot be isometrically embedded into $\R^3$ via $C^2$ maps; see Hilbert--Cohn Vossen \cite{hilbert}.

The third approach to the existence of isometric immersions is via the analysis of the Gauss--Codazzi equations. Temporarily let us only present the key ideas here; the detailed derivations can be found in \S 2. Let $(L,M,N)$ be the components of the normalised second fundamental form of $f:(\M,g) \map (\R^3, \geu)$. The Codazzi equations are the following first-order PDE system:
\begin{equation}\label{codazzi eqs}
\begin{cases}
M_x - L_y = \G^2_{22} L - 2\G^2_{12} M + \G^2_{11}N,\\
N_x - M_y = -\G^1_{22} L + 2\G^{1}_{12} M - \G^1_{11} N,
\end{cases}
\end{equation}
where $\{\G^{i}_{jk}\}_{1\leq i,j,k \leq 2}$ denote the Christoffel symbols of the Levi--Civita connection on $(\M,g)$. In addition, the Gauss equation is a zeroth-order quadratic relation on $\{L,M,N\}$:
\begin{equation}\label{gauss eq}
LN-M^2 = \kappa.
\end{equation}
By the fundamental theorem of surfaces (see Eisenhart \cite{eisenhart}), for smooth metrics the solubility of the Gauss--Codazzi equations \eqref{codazzi eqs} and \eqref{gauss eq} are equivalent to the existence of isometric immersions. This has been extended in \cite{mardare1, mardare2} by Mardare to the case of $g \in L^p$ with $p > \dim(\M)$: It is shown that the existence of {\em weak solutions} of the Gauss--Codazzi system --- in the sense of distributions --- is equivalent to the existence of $W^{1,p}_{\rm loc}$ isometric immersions; also see Chen--Li \cite{cl1} for a geometric proof in the arbitrary dimension/co-dimension case.

In the recent years, there have been rapid developments in the study the isometric immersions of surfaces with negative $\kappa$ via the aforementioned third approach, {\it i.e.}, by analysing the Gauss--Codazzi equations \eqref{gauss eq}\eqref{codazzi eqs}. Chen--Slemrod--Wang \cite{cswb} (also see the exposition \cite{cswc}) first observed that the Codazzi equations \eqref{codazzi eqs} can be viewed as a system of {\em hyperbolic balance laws} for $\{L,M\}$, once we substitute $N$ by $(\kappa+M^2)/L$ using the Gauss equation \eqref{gauss eq}. In this formulation, if one suitably interprets the geometric quantities in terms of fluid mechanics, the cases of $\kappa>0$, $\kappa=0$ and $\kappa < 0$ correspond to the subsonic, sonic and supersonic fluid flows. In particular, in the supersonic case $(\kappa < 0)$ the resulting system of balance laws is  {\em strictly hyperbolic}, which enables us to prove the existence of weak solutions using tools from hyperbolic PDEs. Indeed, by applying the method of {\em vanishing viscosity} and the theory of {\em invariant regions}, which had been exploited in the works of Morawetz on transonic flows (\cite{morawetz1, morawetz2, morawetz3}), Chen--Slemrod--Wang established the global existence of $C^{1,1}$ isometric immersions of a one-parameter family $\{g_\beta\}$ of  negatively curved metrics, where $g_{\sqrt{2}}$ corresponds to the metric of the classical catenoid. For this reason, the family $\{g_\beta\}$ may be termed as ``generalised catenoids''. Subsequently, using another closely related system of balance laws, Cao--Huang--Wang \cite{chw1} further proved the  existence of $C^{1,1}$ isometric immersions of a family of  ``generalised helicoids'', which includes the usual helicoid as a special case. Moreover,  employing another tool from the hyperbolic balance laws, the {\em Lax--Friedrich scheme}, Cao--Huang--Wang \cite{chw2} proved the global existence of $C^{1,1}$ isometric immersions of the metrics of the form $g=E(y)\dd x^2 + \dd y^2$, under the conditions $\log (E(y)^2\sqrt{-\kappa})$ being a non-increasing $C^{1,1}$ function, $E''(y)=-\kappa E(y)$, plus some mild conditions on the initial data ($y=0$). In particular, \cite{chw2} extended the results in \cite{chw1} on the ``generalised helicoids''. Furthermore, exploiting the theory of BV solutions of balance laws ({\it cf.} Dafermos \cite{dafermos}), Christoforou \cite{christofourou} established the existence of $C^{1,1}$ isometric immersions of the conformal metrics $g_{c,q^*}=(\cosh(cx))^{2/({q^*}^2-1)}(\dd x^2 + \dd y^2)$, where $c>0$, $q>1$ are constants, with conditions on $q^*$ and/or the initial data ($y=0$). The collection $\{g_{c,q^*}\}$ neither contains or is contained in the family of ``generalised catenoids'' as in Chen--Slemrod--Wang \cite{cswb} or Cao--Huang--Wang \cite{chw1}. The existence result in \cite{christofourou} requires further {\it a priori} stability assumptions,  entailed by the ``BV framework'' formulated in the paper.

	To summarise, in the recent works \cite{chw1, chw2, cswb, cswc, christofourou}, various negatively curved metrics of the ``generalised catenoid type'' or the ``generalised helicoid type'' have been proved to possess global $C^{1,1}$ isometric immersions into $\R^3$. 
	
	The main contribution of the current paper is to establish the global or local existence of $C^{1,1}$ isometric immersions for several more families of such metrics. This is done by further exploiting the Gauss--Codazzi equations \eqref{gauss eq}\eqref{codazzi eqs} via the theory of hyperbolic balance laws, as well as exploring the detailed structures of relevant geometric quantities, {\it e.g.}, the Gauss curvature, the metric components and the Christoffel symbols. Before subsequent developments, let us note that the metrics considered in this paper are still far from being general: special structures of the metrics are essential to control the {\em source terms} of the associated hyperbolic balance laws, which play a crucial role in the theory of these PDEs; see Dafermos \cite{dafermos}.

	The remaining parts of the paper is organised as follows. In \S 2 we derive Gauss--Codazzi equations and recast them into suitable hyperbolic balance laws. Next, in \S 3 we briefly discuss the theory of compensated compactness and invariant regions, as well as the application to the Gauss--Codazzi system. Then, in \S \S 4--6 respectively, we establish the existence of $C^{1,1}$ isometric immersions for three families of metrics: the standard hyperbolic (Lobachevsky) plane, ``helicoid-type metrics'' and ``generalised Enneper metrics''. In \S 7 we discuss the non-existence of invariant regions of ``reciprocal-type metrics''. Finally, in \S 8 we conclude with a few remarks.

\section{The PDEs}

\subsection{Gauss--Codazzi equations}
In this paper we let $(\M,g)$ be a $2$-dimensional Riemannian manifold, {\it i.e.}, a surface, with Riemannian metric $g$. The regularity of $g$ is assumed to be at least Lipschitz, but is not required to be smooth or analytic in general. We consider the isometric immersion $f: (\M,g) \map (\R^3, \geu)$: this means that the Euclidean metric pulled back via $f$ coincides with the given metric $g$ on $\M$, namely $f^*\geu = g$. It is equivalent to Eq.\,\eqref{nash eq}.

If an isometric immersion $f: (\M,g) \map (\R^3, \geu)$ exists, for each point $x \in \M$, the tangent space $T_{f(x)}\R^3$ splits into orthogonal vector spaces:
\begin{equation}\label{split}
T_{f(x)}\R^3 \cong T_x\M \bigoplus [T_x\M]^\perp.
\end{equation}
We write $T\M^\perp$ for the normal bundle, {\it i.e.}, the vector bundle over $\M$ with fibres $T_x\M$. Denoting by $\barna$ the trivial Levi--Civita connection on $T\R^3$ and $\na$ the Levi--Civita connection on $T\M$, we  can define $\two: \G(T\M) \times \G(T\M) \map \G(T\M^\perp)$ via
\begin{equation}\label{def of II, eq}
\two(X,Y) := \barna_{X}Y - \na_XY,
\end{equation}
where $\G(\E)$ for a vector bundle $\E$ denotes the space of sections of $\E$. This tensor field $\two$ takes values in the normal bundle, and it describes the manner in which $\M$ is immersed in the ambient Euclidean space. This extrinsic quantity is known as the {\em second fundamental form}. For the immersion $f$ of surfaces into $\R^3$, $T\M^\perp$ is 1-dimensional, so $\two$ can be viewed as a $\R$-valued $2 \times 2$ matrix field (namely, a section of $T^*\M\bigotimes T^*\M$), or equivalently, a field of quadratic forms (namely, a section of $T\M \bigotimes T\M$). Therefore, by a slightly abusive notation we may write
\begin{equation}
\two = \begin{bmatrix}
h_{11} & h_{12}\\
h_{21} & h_{22}
\end{bmatrix} \qquad \text{ or } \qquad  \two = h_{11} \dd x^2 + 2h_{12}\dd x\dd y + h_{22} \dd y^2,
\end{equation}
where $h_{12}=h_{21}$ in view of the definition of $\two$ in Eq.\,\eqref{def of II, eq} and the basic properties of $\barna$ and $\na$. In addition, we also write 
\begin{equation}
\two (X,Y,\eta) := \geu(\two(X,Y), \eta) \qquad \text{for } X,Y \in \G(T\M); \, \eta \in \G(T\M^\perp).
\end{equation}

The Gauss and Codazzi equations express the orthogonal splitting of the zero Riemann curvature of $(\R^3, \geu)$  along Eq.\,\eqref{split}. Let $X,Y,Z,W \in \G(T\M)$ be tangential vector fields and $\eta \in \G(T\M^\perp)$ be a normal vector field. Also, denote the inner products induced by the metrics $g$ or $\geu$ both by $\langle\cdot,\cdot\rangle$, and write $R$ for the Riemann curvature tensor on $\M$. With these preparations, the Gauss equation reads:
\begin{equation}\label{gauss, global}
R(X,Y,Z,W) = \langle \two (X,Z), \two (Y,W) \rangle - \langle \two (X,W), \two (Y,Z) \rangle,
\end{equation}
and the Codazzi equations are as follows:
\begin{equation}\label{codazzi, global}
\barna_Y \two(X,Z,\eta) - \barna_X \two(Y,Z,\eta) = 0.
\end{equation}
Due to our regularity assumptions on $g$, Eqs.\,\eqref{gauss, global}\eqref{codazzi, global} should be understood in the sense of distributions. Here $g$ (hence $R$, due to Gauss's Theorema Egregium) are given functions and $\two$ consist of the unknowns. For the geometric formulations above, we refer to \S 6 in do Carmo \cite{docarmo} or \S 2 in  Chen--Li \cite{cl1} for detailed discussions.

\smallskip
To proceed, let us derive the form of Gauss--Codazzi equations in local coordinates, namely Eqs.\,\eqref{gauss eq}\eqref{codazzi eqs}. These equations are regarded as well-known in Han--Hong \cite{hh} and the aforementioned works \cite{chw1, chw2, cswb, cswc, christofourou}; nevertheless, it is not easy to find in the literature a careful derivation from the first principles, {\it i.e.}, the global Eqs.\,\eqref{gauss, global}\eqref{codazzi, global}. For the convenience of the readers we shall present the derivation below. 

Let $\{\p_1, \p_2\}$ be a local coordinate frame on $T\M$, and $\p_3$ be the unit normal vector field in $T\M^\perp$. Taking $X= \p_1, Y= \p_2, Z=\p_1$ and $W=\p_2$, the Gauss equation \eqref{gauss, global} gives us
\begin{equation}\label{gauss'}
R_{1212} = h_{11}h_{22} - (h_{12})^2,
\end{equation}
where $R_{ijkl}:=R(\p_i, \p_j, \p_k, \p_l)$. For the Codazzi equation, there are two independent choices of coordinates: $(i,j,k)=(1,2,1)$ and $(i,j,k)=(1,2,2)$. They lead respectively to
\begin{equation}\label{codazzi'}
\begin{cases}
\barna_{\p_2}\two(\p_1, \p_1, \p_3) - \barna_{\p_1} \two(\p_2, \p_1, \p_3) = 0,\\
\barna_{\p_2}\two(\p_1, \p_2, \p_3) - \barna_{\p_1}\two(\p_2, \p_2, \p_3)=0.
\end{cases}
\end{equation}
By the Leibniz rule, we have the identity
\begin{align}\label{a}
\barna_{\p_i}\two(\p_j, \p_k, \p_\alpha) &= \p_i \two(\p_j, \p_k \p_\alpha) - \two(\na_{\p_i}\p_j, \p_k, \p_\alpha) \\
&\qquad - \two(\p_j, \na_{\p_i}\p_k, \p_\alpha) - \two(\p_j, \p_k, \na^\perp_{\p_i}\p_{\alpha}),
\end{align}
where $\na^\perp$ is the projection of $\barna$ onto $T\M^\perp$. In our case, the co-dimension of the immersion is $1$; hence $\p_i\langle\p_\alpha, \p_\alpha\rangle = 2\langle\na^\perp_{\p_i}\p_\alpha, \p_\alpha\rangle = 0$, which forces 
\begin{equation*}
\two(\p_j, \p_k, \na^\perp_{\p_i}\p_\alpha) = 0.
\end{equation*}
Thus, together with the definition for the Christoffel symbols
\begin{equation}
\na_{\p_i}\p_j = \G^l_{ij} \p_l,
\end{equation}
Eq.\,\eqref{a} becomes
\begin{equation}\label{b}
\barna_{\p_i}\two(\p_j, \p_k, \p_\alpha) = \p_i h_{jk} - \big\{\G^p_{ij}h_{pk} + \G^p_{ik}h_{pj}\big\}.
\end{equation}
Substituting Eq.\,\eqref{b} into Eq.\,\eqref{codazzi'}, after some  cancellations we arrive at the following:
\begin{equation}\label{codazzi''}
\begin{cases}
\big\{\p_2h_{11} - \p_1 h_{12}\big\} -\G^1_{12}h_{11} + (\G^1_{11} - \G^2_{12})h_{12} + \G^2_{11}h_{22}=0,\\
\big\{\p_2h_{12} - \p_1h_{22}\big\} -\G^1_{22}h_{11} + (\G^1_{12}-\G^2_{22})h_{12} + \G^2_{12}h_{22} = 0.
\end{cases}
\end{equation}

Up to now, we have derived Eq.\,\eqref{gauss'} and \eqref{codazzi''}, which are equivalent to the Gauss and Codazzi equations, respectively. Let us introduce the following {\em normalisation}. For brevity, write
\begin{equation*}
|g|:=\det(g),
\end{equation*}
which is strictly positive. Then, by definition, the {\em Gauss curvature} is
\begin{equation}\label{kappa}
\kappa := \frac{R_{1212}}{|g|}.
\end{equation}
Moreover, we can then define the functions $L,M,N$ on $\M$ by the following:
\begin{equation}
\frac{1}{|g|} \, \two = \frac{1}{|g|} \begin{bmatrix}
h_{11}&h_{12}\\
h_{21}&h_{22}
\end{bmatrix} = \begin{bmatrix}
L & M\\
M&N
\end{bmatrix}.
\end{equation}
So
\begin{equation}\label{c}
\p_i (h_{11}, h_{12}, h_{22})^\top = \sqrt{|g|}\bigg\{\p_i + \frac{\p_i|g|}{2|g|}\bigg\} (L,M,N)^\top.
\end{equation}
On the other hand, using the formula for Christoffel symbols in the local coordinates:
\begin{equation*}
\G^i_{jk} = \frac{1}{2} g^{il} \big\{\p_j g_{kl} + \p_kg_{lj} - \p_l g_{jk}\big\}
\end{equation*}
as well as identities for the derivative of the logarithmic of matrices, we can deduce that
\begin{align}\label{d}
\G^1_{12} + \G^2_{22} &= \frac{1}{2} \Big\{g^{11}\p_2g_{11} + g^{22}\p_2g_{22} + 2g^{12}\p_2g_{12}\Big\}\nonumber\\
&=\frac{1}{2} {\rm trace} \big(\p_i \log g\big) \nonumber\\
&= \p_i\big\{\log (\det \, g)\big\} = \frac{\p_i |g|}{|g|},
\end{align}
thanks to the positive definiteness of $g$. Putting everything together, Eqs.\,\eqref{kappa} and \eqref{gauss'} become the local form of the Gauss equation:
\begin{equation}\label{G}
\kappa = LN-M^2,
\end{equation}
and Eqs.\,\eqref{c}, \eqref{d} applied to Eq.\,\eqref{codazzi''} give us the Codazzi equations:
\begin{equation}\label{C1}
\p_1 M - \p_2 L = \G^2_{22}L-2\G^2_{12}M+\G^2_{11}N,
\end{equation} 
\begin{equation}\label{C2}
\,\,\,\,\p_1 N - \p_2 M = -\G^1_{22}L + 2\G^1_{12}M -\G^1_{11}N,
\end{equation}
which reproduce Eqs.\,\eqref{gauss eq}\eqref{codazzi eqs}.

In the sequel, we shall refer to Eqs.\,\eqref{G}\eqref{C1}\eqref{C2} as the Gauss and Codazzi equations.
We also write $\{\p_1, \p_2\}$ and $\{\p_x, \p_y\}$ interchangeably.

\subsection{Hyperbolic balance laws}
In this subsection, let us transform the Gauss and Codazzi equations to the associated system of hyperbolic balance laws  following Cao--Huang--Wang \cite{chw1, chw2}. To make the paper self-contained, we shall describe the main steps of the transform, with a remark on the necessity of the negative curvature condition $\kappa < 0$.

{\bf Step 1: Fluid formulation.} First of all, let us introduce
\begin{equation}\label{def gamma}
\gamma:=\sqrt{-\kappa},
\end{equation}
which is a positive $C^{1,1}$ function. We further normalise
\begin{equation}\label{tilde LMN}
\begin{bmatrix}
\tl & \tm\\
\tm & \tn
\end{bmatrix} = \frac{1}{\gamma} \begin{bmatrix}
L&M\\
M&N
\end{bmatrix},
\end{equation}
so that the {\em normalised Gauss equation} becomes
\begin{equation}\label{normalised G}
\tl\tn-{\tm}^2 = -1.
\end{equation}
The {\em normalised Codazzi equations} are computed in Eq.\,(2.7) of Cao--Huang--Wang \cite{chw2}:
\begin{eqnarray}
&&\tm_x - \tl_y = \tg^2_{22}\tl - 2\tg^2_{12}\tm + \tg^2_{11} \tn,\label{normalised C1} \\
&&\tn_x-\tm_y = -\tg^1_{22}\tl + 2\tg^1_{12}\tm - \tg^1_{11}\tn. \label{normalised C2}
\end{eqnarray}
with $\tg^i_{jk}$, the normalised Christoffel symbols, given by
\begin{align}
&\tg^{2}_{22} = \G^2_{22} + \frac{\gamma_y}{\gamma}, \qquad \tg^2_{12}= \G^2_{12} + \frac{\gamma_x}{2\gamma}, \qquad \tg^2_{11}=\G^2_{11},\nonumber\\ &\tg^1_{22} = \G^1_{22}, \qquad
\tg^1_{12} = \G^1_{12} + \frac{\gamma_y}{2\gamma}, \qquad \tg^1_{11} = \G^1_{11} + \frac{\gamma_x}{\gamma}.
\end{align}

Now, let us introduce the ``fluid variables'':
\begin{equation}\label{rho, m}
U:= \begin{bmatrix}
\tl\\
-\tm
\end{bmatrix} \equiv \begin{bmatrix}
\rho\\
m
\end{bmatrix},
\end{equation}
where $\rho$ is the density and $m$ 
the momentum of a 1-dimensional compressible fluid. Thus, the normalised Gauss equation \eqref{normalised G} gives us
\begin{equation}
\tn = \frac{m^2-1}{\rho},
\end{equation}
which recasts the normalised Codazzi equations \eqref{normalised C1}\eqref{normalised C2} to the following balance law:
\begin{equation}\label{balance law}
U_y + [\FF(U)]_x = \GG(U),
\end{equation}
where
\begin{equation}
[\FF(U)]_x = \na \FF(U) \cdot U_x = \begin{bmatrix}
m_x\\
\frac{2mm_x}{\rho} - \frac{m^2-1}{\rho^2}\rho_x,
\end{bmatrix}
\end{equation}
and the right-hand side
\begin{equation}
\GG(U) = \begin{bmatrix}
-\tg^2_{22}\rho - 2\tg^2_{12}m - \tg^2_{11} \big(\frac{m^2-1}{\rho}\big)\\
-\tg^1_{22}\rho - 2\tg^1_{12}m -\tg^1_{11}\big(\frac{m^2-1}{\rho}\big)
\end{bmatrix}. 
\end{equation}
The matrix $\na \FF(U)$, where $\na$ is the gradient of $\FF$ with respect to $U$, is
\begin{equation}
\na \FF(U) = \begin{bmatrix}
0 & 1\\
-\frac{m^2-1}{\rho^2} & \frac{2m}{\rho}
\end{bmatrix},
\end{equation}
whose eigenvalues are
\begin{equation}
\lambda_\pm = \frac{m \pm 1}{\rho}.
\end{equation}

The above calculation for eigenvalues requires
\begin{equation*}
\rho \neq 0 \Longleftrightarrow \tl \neq 0 \Longleftrightarrow L \neq 0.
\end{equation*}
As $g$ is symmetric and positive definite, this is automatically satisfied. 

\smallskip
{\bf Step 2: Parabolic regularisation.} Now let us introduce the parabolic regularistion for the normalised Gauss--Codazzi equations, {\it i.e.,} the viscous approximation equations in terms of the fluid variables. For each $\e>0$ sufficiently small, we choose the regularisation as follows:
\begin{eqnarray}
&&(\rho^\e)_y + (m^\e)_x = \e(\rho^\e)_{xx}-\tg^2_{22}(\rho^\e) - 2\tg^2_{12}m^\e - \tg^2_{11}\Big(\frac{(m^\e)^2-1}{\rho^\e}\Big), \label{epsilon C1} \\
&&(m^\e)_y + \Big(\frac{(m^\e)^2-1}{\rho^\e}\Big)_x = \e(m^\e)_{xx}-\tg^1_{22}(\rho^\e) - 2\tg^1_{12}m^\e - \tg^1_{11}\Big(\frac{(m^\e)^2-1}{\rho^\e}\Big). \label{epsilon C2}
\end{eqnarray}

Then, introducing 
\begin{equation}
u:=\frac{m}{\rho}, \qquad v:= \frac{1}{\rho},
\end{equation}
(the superscript $\e$ will be dropped for simplicity), one can write
\begin{equation*}
\rho = \frac{1}{v}, \qquad m = \frac{u}{v}, \qquad \tn = \frac{u^2-v^2}{v}.
\end{equation*}
The regularised equations \eqref{epsilon C1}\eqref{epsilon C2} become
\begin{align} \label{epsilon C1'}
v_y -vu_x + uv_x = \e v_{xx}- \frac{2\e(v_x)^2}{v} + v\tg^2_{22} + 2\tg^2_{12}uv + \tg^2_{11}(u^2-v^2)v,
\end{align}
as well as
\begin{align}\label{epsilon C2'}
&u_y -vv_x +uu_x \nonumber\\
=\,& \e u_{xx}- \frac{2\e u_x v_x}{v} + (\tg^2_{22}-2\tg^1_{12})u + 2\tg^2_{12}u^2 + \tg^1_{11} (v^2-u^2) + \tg^2_{11}u(u^2-v^2) -\tg^1_{22}.
\end{align}

\smallskip
{\bf Step 3: Decoupling.} Taking the sum and the difference of Eqs.\,\eqref{epsilon C1'} and \eqref{epsilon C2'}, one can simplify the first- and second-order terms of $u$ and $v$. This amounts to introducing the {\em Riemann invariants} ({\it cf}. Part III of Smoller \cite{smoller}):
\begin{equation}
w:= u+v, \qquad z：= u-v.
\end{equation}
Thus, let us rewrite the regularised system in terms of the rotated coordinates $\{w,z\}$:
\begin{eqnarray}
&&w_y + zw_x =\e w_{xx} - \frac{2\e v_x w_x}{v} + \sone,\label{wz 1}\\
&& z_y + wz_x = \e z_{xx} - \frac{2\e v_xz_x}{v} + \stwo, \label{wz 2}
\end{eqnarray}
where the source terms $\sone$ and $\stwo$, viewed as functions of $(u, v)$, are inhomogeneous cubic polynomials in $(u,v)$ with coefficients involving linear combinations of $\tg^i_{jk}$:
\begin{align}\label{S1}
\sone &= -\tg^1_{22} \nonumber\\
&\quad  + (\tg^2_{22}-2\tg^1_{12})u + \tg^2_{22}v \nonumber\\
&\quad  + (2\tg^2_{12} - \tg^1_{11})u^2 + \tg^1_{11}v^2 + 2\tg^2_{12} uv \nonumber\\
&\quad + \tg^2_{11}u(u-v)(u+v) + \tg^2_{11}v(u-v)(u+v);
\end{align}
\begin{align}\label{S2}
\stwo &= -\tg^1_{22} \nonumber\\
&\quad + (\tg^2_{22}-2\tg^1_{12})u - \tg^2_{22}v \nonumber\\
&\quad  + (2\tg^2_{12} - \tg^1_{11})u^2 + \tg^1_{11}v^2 - 2\tg^2_{12} uv \nonumber\\
&\quad + \tg^2_{11}u(u-v)(u+v) - \tg^2_{11}v(u-v)(u+v).
\end{align}

Throughout in the sequel we shall work with the system of Eqs.\,\eqref{wz 1}\eqref{wz 2} for the Riemann invariants $w$ and $z$, with the source terms given by  $\sone$ and $\stwo$ in Eqs.\,\eqref{S1}\eqref{S2}. In full generalities, no further simplifications are available for $\sone$ and $\stwo$.

\smallskip
{\bf Step 4: Formulae of $\tg^i_{jk}$ for diagonal metrics.}
After the derivation of the  hyperbolic balance law (Eqs.\,\eqref{wz 1}\eqref{wz 2}), let us remark that, under the following working assumption of this paper, the coefficients of $\sone$ and $\stwo$ can be evaluated easily. Throughout this paper $g$ is assumed to be diagonal (it is natural because $g$ is always diagonalisable as a positive definite symmetric matrix):
\begin{equation}
g = \begin{bmatrix}
E(x,y) & 0\\
0 & G(x,y)
\end{bmatrix}.
\end{equation}
In this case, $\G^i_{jk}$ and thus $\tg^i_{jk}$ are given as follows:
\begin{align}\label{Gamma}
&\G^1_{11} = \frac{E_x}{2E}, \qquad \G^1_{12} = \frac{E_y}{2E}, \qquad \G^1_{22} = -\frac{G_x}{2E},\nonumber\\
&\G^2_{11} = -\frac{E_y}{2G}, \qquad \G^2_{12} = \frac{G_x}{2G}, \qquad \G^2_{22} = \frac{G_y}{2G};
\end{align}
\begin{align}\label{tilde Gamma}
&\tg^1_{11} = \frac{E_x}{2E} + \frac{\gamma_x}{\gamma}, \qquad \tg^1_{12} = \frac{E_y}{2E} + \frac{\gamma_y}{2\gamma}, \qquad \tg^1_{22} = -\frac{G_x}{2E},\nonumber\\
&\tg^2_{11} = -\frac{E_y}{2G}, \qquad \tg^2_{12} = \frac{G_x}{2G} + \frac{\gamma_x}{2\gamma}, \qquad \tg^2_{22} = \frac{G_y}{2G} + \frac{\gamma_y}{\gamma}.
\end{align}

\subsection{$\kappa <0$ is essential} We now discuss the necessity of the assumption of negative curvature to our approach. Suppose $\kappa \geq 0$ in some region $\M' \subset \M$ and that the normalisation factor, again denoted by $\gamma$, where ${\gamma}=\sqrt{\kappa} \in C^{1,1}(\M')$. Also, define the normalised second fundamental form $\{\tl, \tm, \tn\}$ as in Eq.\,\eqref{tilde LMN}, and introduce the fluid variables $\{\rho, m\}$ as in Eq.\,\eqref{rho, m}. (Here the sign of $m$ is not essential; what is important is that we do not introduce further derivatives of $\{\tl, \tm, \tn\}$ in the fluid formulation.) Then, notice that the formulae for the normalised Codazzi equations \eqref{normalised C1}\eqref{normalised C2} remain unchanged; thus, in the balance law \eqref{balance law}, we have
\begin{equation*}
[\FF(U)]_x = \begin{bmatrix}
m_x\\
\frac{2m}{\rho} m_x - \frac{m^2+1}{\rho^2}\rho_x
\end{bmatrix},
\end{equation*} 
hence
\begin{equation}
\na \FF(U) = \begin{bmatrix}
 0 & 1\\
 -\frac{m^2+1}{\rho^2} & \frac{2m}{\rho}
\end{bmatrix}.
\end{equation}

However, the eigenvalues of the matrix $\na F(U)$ in the region $\M'$ can be computed as follows:
\begin{equation}
\lambda_\pm = \frac{1}{2}\bigg\{\frac{2m}{\rho} \pm \sqrt{\Big(\frac{2m}{\rho}\Big)^2 - 4\Big(\frac{m^2+1}{\rho^2}\Big) } \bigg\} = \frac{1}{2} \bigg\{\sqrt{-\frac{4}{\rho^2}} \pm \frac{2m}{\rho}\bigg\},
\end{equation}
which is not a real value. In view of the definition of strict hyperbolicity ({\it cf.} Dafermos \cite{dafermos}), we have proved:
\begin{proposition}\label{propn: kappa < 0}
The balance law \eqref{balance law} associated to the fluid formulation of the normalised Gauss and Codazzi equations \eqref{normalised G}\eqref{normalised C1}\eqref{normalised C2} in $\S 2.2$ is strictly hyperbolic if and only if the Gauss curvature $\kappa$ of the manifold $(\M, g)$ is negative.
\end{proposition}

The above proposition  explains why we need to restrict to the case $\kappa <0$ in \cite{cswb,cswc,chw1,chw2} and this paper. It echoes that in the Darboux equation formulation for the isometric immersion problem, $\kappa > 0$ ensures the ellipticity of Eq.\,\eqref{darboux eq}.

In passing, we comment that the  isometric immersion problems have also been transformed to hyperbolic PDEs by taking sufficiently many derivatives of the immersion maps. In the work \cite{han} by Han, the local existence of $C^{r-6}$ isometric immersions near the origin $(x,y)=(0,0)$ for $g \in C^r$ is proved by analysing the Darboux equation \eqref{darboux eq}, for $r \geq 9$, $\kappa(0)=0$ and $\na\kappa(0) \neq 0$. 
Han \cite{han} simplified the earlier arguments due to Lin \cite{lin}. For this purpose, one considers the ansatz $u=u_0 + \e^6 \widetilde{u}$ of the classical solution to the Darboux equation \eqref{darboux eq}, where $u_0$ is an approximate background solution. The crucial observation in \cite{han} is that certain {\em higher order derivatives of $\widetilde{u}$} satisfy a first-order positive symmetric hyperbolic system, which can be solved by carrying out usual energy estimates. Let us also mention that in \cite{asian}, Chen--Clelland--Slemrod--Wang--Yang obtained a simplified proof of the local existence of a smooth isometric embedding of a smooth 3-dimensional Riemannian manifold with non-zero Riemannian curvature tensor into 6-dimensional Euclidean space, also by reducing to first-order positive symmetric hyperbolic systems. We also refer to Efimov \cite{efimov}, Tunitski\u{i} \cite{russian1}, Poznyak \cite{russian2}, Rozendorn \cite{rozendorn} and the  references therein for the pioneering works on the isometric immersions of surfaces of negative curvatures by the Russian school.

\section{Compensated Compactness}
In this section we discuss our central analytic tool --- compensated compactness.

Introduced by Murat \cite{murat} and Tartar \cite{tartar}, the theory of compensated compactness has been extensively employed in  non-linear hyperbolic PDEs and  applications to gas dynamics and elasticity; see Ball \cite{ball}, Chen--Slemrod--Wang \cite{cswa}, Dafermos \cite{dafermos}, Ding--Chen--Luo \cite{dcl}, DiPerna \cite{dp1, dp2}, Evans \cite{evans}, Morawetz \cite{morawetz1, morawetz2, morawetz3}, Tartar \cite{tartar} and the references cited therein. Recently, its applications to the isometric immersions problem have been discovered by Chen--Slemrod--Wang \cite{cswc, cswd} and studied by Cao--Huang--Wang \cite{chw1, chw2} and Christofourou \cite{christofourou}, among others.  

In this paper, as in \cite{cswc, chw1, chw2, christofourou}, we prove the existence of weak solutions to the Gauss--Codazzi equations by the method of vanishing artificial viscosity, for which the vanishing viscosity limit is obtained in virtue of the compensated compactness method. The uniform $L^\infty$ estimate that guarantees the existence of the vanishing viscosity limit is established by the method of {\em invariant regions}. Finally, the existence of weak solutions to the Gauss--Codazzi equations imply the existence of $C^{1,1}$ isometric immersions of the corresponding surfaces into $(\R^3, \geu)$, due to Mardare \cite{mardare1, mardare2} and Chen--Li \cite{cl1}.

\subsection{Uniform $L^\infty$ estimates via invariant regions}

To begin with, we shall analyse the system of parabolic equations \eqref{wz 1}\eqref{wz 2} for the Riemann invariants $w$ and $z$ (see Step 3 in \S 2.2). Recall that these PDEs are obtained as the parabolic regularisations of the original hyperbolic system of Codazzi equations, via adding to the Codazzi equations the terms of ``artificial 
viscosity'', namely $\e v_{xx}\equiv \e(v^\e)_{xx}$ and $\e u_{xx}\equiv\e(u^\e)_{xx}$. Our goal is to show that, as $\e \map 0^+$, $\{u^\e, v^\e\}$ converges to the weak solution of Eqs.\,\eqref{wz 1}\eqref{wz 2}. 

Introduce the state (column) vector field
\begin{equation}
U \equiv U^\e := (w^\e, z^\e)^\top \equiv (w,z)^\top;
\end{equation}
we view $U$ as the map $(u,v)^\top \mapsto (w,z)^\top$, a transform from $\R^2$ to itself. For simplicity let us systematically drop the superscripts $\e$. Then, with
\begin{equation}
\source=\source(u,v):= \big(\sone,\stwo\big)^\top
\end{equation}
the system \eqref{wz 1}\eqref{wz 2} can be recast into the following balance law:
\begin{equation}\label{system of balance law for w,z}
U_y =\e U_{xx} + \underbrace{\begin{bmatrix}
-2\e v_xv^{-1}-z &0\\
0 & -2\e v_xv^{-1}-w
\end{bmatrix}}_{=: M(u,v)} U_x + \source.
\end{equation}

Now let us explain the method of invariant regions: our presentation  follows closely the classical paper \cite{ccs} by Chueh--Conley--Smoller. Consider the PDE system for $U=U(x,t):\R^m \times \R \map \R^n$
\begin{equation}\label{generic system}
U_t = \e D\cdot\Delta U + \sum_{j=1}^m M^j U_{x_j} + \source,
\end{equation}
where $D$ is a positive definite $n\times n$ matrix, $\{M^1, \ldots, M^m\}$ are $n\times n$ matrices and $\source:\R^m \times \R \map \R^n$. We say that $Q \subset {\rm Image}(U) \subset \R^n$ is an {\em invariant region} of Eq.\,\eqref{generic system} if $U_0(x):=U(x,t_0) \in Q$ implies that $U(x,t)\in Q$, whenever $t_0 \leq t \leq T_\star$ ($T_\star$ = the lifespan of the solution). Then, the following characterisation of invariant regions of Eq.\,\eqref{generic system} is at hand:
\begin{proposition}[Theorem 4.4 in \cite{ccs}] \label{propn: inv region}
In the above setting, consider
\begin{equation}
Q:=\bigcap_{i=1}^d \Big\{z\in\R^n: G_i(z) \leq 0, \text{ $G_i:{\rm Image} (U) \map \R$ are smooth functions}\Big\}.
\end{equation}
Then $Q$ is an invariant region for Eq.\,\eqref{generic system} for all $\e>0$ if and only if the following three conditions holds:
\begin{enumerate}
\item
$\na_{U_0} G_i$ is a left eigenvector for $D$ and each $M^j$, $j=1,2,\ldots m$;
\item
For each vector $\eta \in \R^n$, we have $\na^2_{U_0}G_i(\eta,\eta)\geq 0$ whenever $\na_{U_0} G_i(\eta)=0$;
\item
$\na_{U_0} G_i \cdot \source (U_0) \leq 0$.
\end{enumerate}
In the above, $U_0$ is an arbitrary point on $\p Q$.
\end{proposition}

Specialising to our system \eqref{system of balance law for w,z}, let us take $m=1$, $D={\rm Id}$ and $d=4$. The following choice for $\{G_i=G_i(w,z)\}_{i=1}^4$ clearly satisfies Conditions (1)(2) in Proposition \ref{propn: inv region}:
\begin{equation}\label{G_i}
G_1 := w+c_1, \qquad G_2:= z+c_2, \qquad G_3:=-w+c_3,\qquad G_4:=-z+c_4,
\end{equation}
where the constants $c_1,\ldots, c_4$ are chosen such that $Q:=\bigcap_{i=1}^4 \{G_i \leq 0\}$ is a closed square whose sides intersecting the $u$ and $v$-axes at an angle of $\pi/4$. Moreover, in the $(u,v)$-plane let us denote the edges of $Q$ by $e_i:=\{G_i=0\}$ (hence $\p Q = \bigcup_{i=1}^4 e_i$). To fix the notations, we require that the edge $e_1$ connects the top and the right vertices of $Q$ in the $(u,v)$-plane, $e_2$ connects the right and the bottom vertices, $e_3$ the bottom and the left vertices, and $e_4$ the left and the top vertices. In addition, we we $(\text{top, right, bottom, left vertices}) := (\bullet_t, \bullet_r, \bullet_b, \bullet_l)$.

Thus, Proposition \ref{propn: inv region} immediately leads to the following criterion for invariant regions to the Gauss--Codazzi equations:
\begin{proposition}\label{propn: inv regions VIP}
The square $Q$ specified by Eq.\,\eqref{G_i} is an invariant region for the balance law \eqref{system of balance law for w,z} if and only if $\sone \leq 0$ on $e_1$, $\stwo \leq 0$ on $e_2$, $\sone \geq 0$ on $e_3$ and $\stwo \geq 0$ on $e_4$.
\end{proposition}

That is, the uniform $L^\infty$ estimates on the Riemann invariants $w$ and $z$ (which is ensured by the definition of invariant regions) can be deduced if one can draw a square $Q$ in the $(u,v)$-plane such that, on the four sides of it, the source terms $\sone, \stwo$ have correct signs. On the other hand, $\sone, \stwo$ are given explicitly by Eqs.\,\eqref{S1} and \eqref{S2}; in various nice cases, their zero loci are determined by cubic algebraic curves (denoted by $\wp$ in the sequel) in the $(u,v)$-plane. In view of Proposition \ref{propn: inv regions VIP}, in these cases the analytic problem of proving  uniform $L^\infty$ estimates can  be translated to the geometric problem of finding suitable squares $Q$ satisfying the sign conditions prescribed by the relative positions of the zero loci of $\sone, \stwo$.

Before further development, we remark that the method of invariant regions have been employed by Morawetz \cite{morawetz1, morawetz2, morawetz3} in the study of transonic flows (also see Chen--Slemrod--Wang \cite{cswa}), as well as in various reaction-diffusion models (see \cite{ccs} and the references therein).

\subsection{$C^{1,1}$ Weak rigidity of Gauss--Codazzi equations and isometric immersions}

With the uniform $L^\infty$ estimates on the components of second fundamental forms at hand, we can send $\e\map 0^+$ ({\it i.e.}, evaluate the vanishing viscosity limit) and show that the weak limit is still a weak solution to the Gauss--Codazzi equations. Moreover, back to the geometric problem, we can find an isometric immersion whose second fundamental form is prescribed by the above weak limit.

The above result is known as the weak continuity of the nonlinear PDEs (Gauss--Codazzi), or the weak rigidity of isometric immersions. In fact, we have the following Proposition \ref{propn: chen-li}, which is a special case of Theorem 4.1 and Remark 4.1 in Chen--Li \cite{cl1}. In \cite{cl1} the isometric immersion of arbitrary dimensions and co-dimensions are studied. We also mention that the last part of the following proposition had been obtained by Mardare in \cite{mardare1, mardare2}.
\begin{proposition}\label{propn: chen-li}
Let $\{L^\e, M^\e, N^\e\}$ be a family of functions on a surface $(\M, g)$ where $g\in W^{1,p}$ for $p \in [2,\infty]$. Suppose that their $L^p$ norms are uniformly bounded on compact subsets of $\M$, and that they approximately satisfy the Gauss--Codazzi equations \eqref{gauss eq}\eqref{codazzi eqs} in the following sense:
\begin{eqnarray}
&& L^\e N^\e - (M^\e)^2 = O_1(\e),\\
&&(M^\e)_x - (L^\e)_y = \G^2_{22} L^\e - 2\G^2_{12} M^\e + \G^2_{11}N^\e + O_2(\e),\\
&&(N^\e)_x - (M^\e)_y = -\G^1_{22} L^\e + 2\G^{1}_{12} M^\e - \G^1_{11} N^\e +O_3(\e),
\end{eqnarray}
where $O_i(\e)\map 0$ in $W^{1,-r}$ for $r>1$ as $\e \map 0^+$, $i\in\{1,2,3\}$, and the above equalities are understood in the sense of distributions. Then, as $\e \map 0^+$, after passing to subsequences, $\{L^\e, M^\e, N^\e\}$ converges weakly in $L^p_{\rm loc}$ to $\{\overline{L}, \overline{M}, \overline{N}\}$, which is a weak solution to the Gauss--Codazzi equations. Moreover, there exists an $C^{1,1}$ $($or $W^{2,\infty})$ isometric immersion of $(\M ,g)$ into $\R^3$, whose second fundamental form is $\begin{bmatrix}
\overline{L} &\overline{M}\\
\overline{M} &\overline{N}
\end{bmatrix}$.
\end{proposition}


In summary, in order to establish $C^{1,1}$ isometric immersions for  negatively curved surfaces, it suffices to find the aforementioned invariant regions, {\it i.e.}, the squares $Q$. This is what we shall do in the subsequent sections for several special classes of metrics. Before carrying out this project, let us single out a simple identity central to the later developments: 

\subsection{A lemma on the derivatives of $\kappa$ and $\gamma$}
The following result enables us to simplify the expressions of the normalised Christoffel symbols $\tg^i_{jk}$ in Eq.\,\eqref{tilde Gamma} for special classes of metrics:
\begin{lemma}\label{lemma: derivatives of kappa and gamma}
Let $\kappa<0$ be the Gauss curvature of $(\M, g)$ and $\gamma:=\sqrt{-\kappa}$ as in Eq.\,\eqref{def gamma}. Then
\begin{equation}
\frac{\na \kappa}{2\kappa} = \frac{\na\gamma}{\gamma}.
\end{equation}
\end{lemma}

\begin{proof}
This is immediate from 
$$\frac{\na\gamma}{\gamma} = \na \log \gamma = \na \log \sqrt{-\kappa}= \frac{\na \log (-\kappa)}{2}  = \frac{\na\kappa}{2\kappa}.$$  \end{proof}

\section{The Standard Hyperbolic Plane}

In this section we establish the $C^{1,1}$ isometric immersions of the standard hyperbolic plane $\mathbb{H}^2 =\{(x,y)\in\R^2: y>0\}$, also known as the Lobachevsky plane. 

Recall that $\mathbb{H}^2$ is a Lie group when one identifies $(x,y)\in\mathbb{H}^2$ with the proper affine functions $g_{(x,y)}:\R\map\R$, $t \mapsto yt+x$, where the group action is the composition of functions ({\it cf}. Chapter 1 in do Carmo \cite{docarmo}). The unique left-invariant metric $g$ on $\mathbb{H}^2$ satisfying $g_{(0,1)} = \geu$ is the standard Lobachevsky metric:
\begin{equation}\label{metric for H2}
g = \begin{bmatrix}
y^{-2} & 0\\
0 & y^{-2}
\end{bmatrix}, \qquad y > 0.
\end{equation}
Its Christoffel symbols are
\begin{equation}
\G^1_{11} = \G^{2}_{12} = \G^1_{22} =0, \quad \G^2_{11}=-\G^1_{12}=-\G^2_{22}=\frac{1}{y}.
\end{equation}
It is well-known that $\mathbb{H}^2$ has constant Gauss curvature $\kappa \equiv -1$; hence $\gamma = -\kappa^2 =1$ and
\begin{equation}
\G^{i}_{jk} = \tg^{i}_{jk} \qquad \text{ for all } i,j,k \in \{1,2\}. 
\end{equation}
Therefore, substituting into Eqs.\,\eqref{S1}\eqref{S2}, we get
\begin{eqnarray}
&&\sone = \frac{1}{y}(u-v)\Big(1+(u+v)^2\Big),\\
&&\stwo = \frac{1}{y}(u+v) \Big(1+(u-v)^2\Big).
\end{eqnarray}
Notice that $y>0$ and the zero loci of $\sone$, $\stwo$ are $\{u=v\}$ and $\{u=-v\}$, respectively. By Proposition \ref{propn: inv regions VIP}, the square $Q$ with the vertices 
$$
\bullet_t = (0, 2a), \quad \bullet_r = (a,a), \quad \bullet_b = (0,0), \quad \bullet_l = (-a,a)
$$
is an invariant region for Eq.\,\eqref{system of balance law for w,z}, where $a>0$ is an arbitrary positive constant. $Q$ can be equivalently characterised by 
\begin{equation}\label{Q for H2}
Q = \{(u,v) \in \R^2: 0 \leq u+v=w \leq 2a, \, -2a \leq u-v=z \leq 0\}. 
\end{equation}
Therefore, we can deduce:
\begin{theorem}\label{thm: H2}
Let $y_0>0$ be arbitrary and denote the initial data by
\begin{equation*}
(u,v)|_{y=y_0} =: (u_0, v_0): \R \map \R^2.
\end{equation*}
Assume that $u_0 \pm v_0$ is bounded on $\R$ and 
\begin{equation}
\inf_{\R} \,(u_0+v_0)>0,\qquad \sup_{\R} \,(u_0-v_0) <0.
\end{equation}
Then there exists an $L^\infty$ weak solution to the Gauss--Codazzi equations \eqref{gauss eq} and \eqref{codazzi eqs} in the half space $\Omega = \R \times (y_0, \infty)$ for the standard metric \eqref{metric for H2} on $\mathbb{H}^2$. Moreover, restricted to the domain $\Omega$, $\mathbb{H}^2$ admits a $C^{1,1}$ isometric immersion into $\R^3$.
\end{theorem}

\begin{proof}
	We first note that the assumptions on $u_0, v_0$ ensure the existence of some point $x_\star \in \R$ such that $\big(u_0(x_\star), v_0(x_{\star})\big) \in Q$ for some $a>0$, where $Q$ is the invariant region described above. Thus we get the uniform $L^\infty$ estimate for $(u,v)$ with the initial data $(u_0, v_0)$ on $\Omega$. On the other hand, we have $(L,M,N) = \big(1/v, -u/v, (u^2-v^2)/v\big)$; so the uniform $L^\infty$ estimate  for $(L,M,N)$ follows, where the superscripts $\e$ has been dropped. Moreover, they satisfy the Gauss--Codazzi equations approximately in the sense of Proposition \ref{propn: chen-li}: this can be seen by adopting verbatim the estimates in Chapter 4 of Cao--Huang--Wang \cite{chw1}. Therefore, the proof is complete by taking the vanishing viscosity limit, in view of Proposition \ref{propn: chen-li}.    \end{proof}

In this above we have established the $C^{1,1}$ isometric immersion of $\mathbb{H}^2$ ``nearly globally'', {\it i.e.}, away from an arbitrarily thin slot $\{(x,y)\in\R^2:0<y<y_0\}$. Thus, various isometric models of $\mathbb{H}^2$, including the Poincar\'{e} disk and the pseudo-sphere, all admit $C^{1,1}$ ``nearly global'' isometric immersions into $\R^3$. In contrast, these models are long known to be not globally $C^2$ embeddable ({\it cf}. Hilbert \cite{hilbert}).

\section{Generalised Helicoid-type metrics}
The goal of this section is two-fold: first, let us prove the existence of isometric immersions of another family of ``generalised helicoid-type'' metrics, which has not been covered by Cao--Huang--Wang \cite{chw1, chw2} and Christoforou \cite{christofourou}; second, we give a further characterisation of the generalised helicoids considered in \cite{chw1, chw2}.

\subsection{Global existence of isometric immersions of some helicoid-type metrics}
As in \cite{chw1, chw2}, let us call the metrics
\begin{equation}
g=\begin{bmatrix}
E(y) & 0 \\
0& 1
\end{bmatrix}
\end{equation}
of the ``{\em helicoid type}''. When $E$ is a suitable hyperbolic cosine function of $y$ (see \S 5.2 below), $g$ is the metric for the classical helicoid, or the generalised helicoids considered in \cite{chw1, chw2}. In this case, Eq.\,\eqref{Gamma} yields that
\begin{equation}
\G^2_{11} = -\frac{E_y}{2}, \qquad \G^1_{12}=\frac{E_y}{2E}, \qquad \text{ other } \G^i_{jk} = 0.
\end{equation}
Then, in view of Eqs.\,\eqref{tilde Gamma}\eqref{S1} and \eqref{S2}, the source terms can be simplified as follows:
\begin{align}\label{source, generalised helicoid}
&\sone = -\frac{E_y}{E} u + \frac{\gamma_y}{\gamma} v - \frac{E_y}{2E} (u-v)(u+v)^2,\nonumber\\
&\stwo= -\frac{E_y}{E} u - \frac{\gamma_y}{\gamma} v - \frac{E_y}{2E} (u-v)^2(u+v).
\end{align}
Moreover, in this case we can compute the Gauss curvature, {\it e.g.}, via Brioschi's formula (see \cite{asg}) as follows:
\begin{equation}\label{e}
\kappa = -\frac{1}{2\sqrt{E}} \frac{\p}{\p y}\Big(\frac{\p_y E}{\sqrt{E}}\Big) = -\frac{E_{yy}}{2E} + \frac{(E_y)^2}{4E^2}.
\end{equation}
Thus, by Lemma \ref{lemma: derivatives of kappa and gamma}, we have
\begin{equation}\label{f}
\frac{\gamma_y}{y} = \frac{-E^2 E_{yyy}+2EE_yE_{yy}-(E_y)^3}{-2E^2E_{yy}+E(E_y)^2} = \frac{EE_{yyy}}{2EE_{yy} - (E_y)^2} - \frac{E_y}{E}.
\end{equation}

From now on let us specialise to the following particular family of metrics, which has not established by the generalised helicoid-type metrics  in \cite{chw1, chw2, cswb, cswc, christofourou}:
\begin{equation}\label{generalised helicoid}
g= \begin{bmatrix}
Ay^2 + By + C & 0\\
0 & 1
\end{bmatrix}, \qquad A>0 \text{ and }  B^2 - 4AC < 0.
\end{equation}
The conditions on $A,B,C$ ensure that $g$ is indeed a Riemannian metric. In this case, Eq.\,\eqref{e} immediately leads to
\begin{equation}
\kappa = \frac{B^2-4AC}{4(Ay^2 +By +C)} < 0,
\end{equation} 
and Eq.\,\eqref{f} becomes
\begin{equation}
\frac{\gamma_y}{\gamma} = -\frac{E_y}{E}.
\end{equation}
Thus, further simplifications of the source terms in Eq.\,\eqref{source, generalised helicoid} are available:
\begin{align}\label{source simplified, generalised helicoid}
&\sone = -\frac{E_y}{E} \Big\{(u+v) + \frac{(u-v)(u+v)^2}{2} \Big\},\\
&\stwo = -\frac{E_y}{E} \Big\{(u-v) + \frac{(u-v)^2(u+v)}{2} \Big\}.
\end{align}

In the case $E_y/E\geq 0$, {\it i.e.}, if we further require $2Ay+B \geq 0$, the signs of $\sone,\stwo$ are determined by the cubic polynomials $\pone,\ptwo\in\R[u,v]$ ({\it cf.} Fig.\,\eqref{figure helicoid}):
\begin{eqnarray}
&&\pone = -(u+v) - \frac{(u-v)(u+v)^2}{2} = -(u+v)\Big[1+\frac{u^2-v^2}{2}\Big],\\
&&\ptwo = (v-u) -\frac{(u-v)^2(u+v)}{2} = (v-u)\Big[1+\frac{u^2-v^2}{2}\Big].
\end{eqnarray}
For $(u,v)\in[0,1]\times[0,1]$ there holds $1+(u^2-v^2)/2 > 0$; by a similar computation as in \S 4, we find that the square $Q$ with vertices
$$
\bullet_t = (0, 2a), \quad \bullet_r = (a,a), \quad \bullet_b = (0,0), \quad \bullet_l = (-a,a)
$$
 is an invariant region whenever $a$ is an arbitrary constant in $(0,1]$. Using analogous arguments as in the proof of Theorem \ref{thm: H2}, we conclude with the following sufficient conditions for the existence of isometric immersions for the helicoid-type metrics:
\begin{theorem}\label{thm: helicoid}
Let $y_0>0$ be arbitrary and denote the initial data by
\begin{equation*}
(u,v)|_{y=y_0} =: (u_0, v_0): \R \map \R^2.
\end{equation*}
Assume that  
\begin{equation}
0<\inf_{\R} \,(u_0+v_0)\leq \sup_{\R} \,(u_0+v_0)\leq 2,\qquad -2\leq\inf_{\R} \,(u_0-v_0)\leq \sup_{\R} \,(u_0-v_0)<0.
\end{equation}
Then there exists an $L^\infty$ weak solution to the Gauss--Codazzi equations \eqref{gauss eq} and \eqref{codazzi eqs} in the half space $\Omega = \R \times (y_0, \infty)$ for the helicoid-type metric
$$
g= \begin{bmatrix}
Ay^2 + By + C & 0\\
0 & 1
\end{bmatrix} \qquad \text{ where } A>0, \, B \geq 0 \text{ and }  B^2 - 4AC < 0.
$$ 
Moreover, restricted to the domain $\Omega$, $g$ admits a $C^{1,1}$ isometric immersion into $\R^3$.
\end{theorem}

On the other hand, in the case $E_y/E \leq 0$, {\it i.e.}, $2Ay+B \leq 0$, the square $Q$ with vertices $\bullet_t = (0,a')$ for $a'>\sqrt{2}$, $\bullet_b = (0,a'')$ for $0\leq a''<\sqrt{2}$, and the other two vertices $\bullet_l, \bullet_r$ lying on the upper branch of the hyperbola $\{(u,v):u^2-v^2+2=0\}$ is an invariant region. Surprisingly, the computation here coincides with that for the ``catenoid-type surfaces'' in \S 3.1 of Cao--Huang--Wang \cite{chw1}; see Eq.\,(3.7) on p1439 therein, with the choice $c=\sqrt{2}$. Therefore, 


\begin{theorem}\label{thm: helicoid 2}
Let $y_0>0$ be arbitrary and denote the initial data by
\begin{equation*}
(u,v)|_{y=y_0} =: (u_0, v_0): \R \map \R^2.
\end{equation*}
Assume that  $u_0 \pm v_0$ is bounded on $\R$, and
\begin{equation}
 \sup_{\R} \,(u_0+v_0)>0,\qquad \inf_{\R} \,(u_0-v_0)<0.
\end{equation}
Then there exists an $L^\infty$ weak solution to the Gauss--Codazzi equations \eqref{gauss eq} and \eqref{codazzi eqs} in the strip $\Omega = \R \times [y_0, -B/2A]$ for the helicoid-type metric
$$
g= \begin{bmatrix}
Ay^2 + By + C & 0\\
0 & 1
\end{bmatrix} \qquad \text{ where } A>0, \, B \leq 0 \text{ and }  B^2 - 4AC < 0.
$$ 
Moreover, restricted to the domain $\Omega$, $g$ admits a $C^{1,1}$ isometric immersion into $\R^3$.
\end{theorem}

\begin{figure}[h]
    \centering
\includegraphics[scale=0.5]{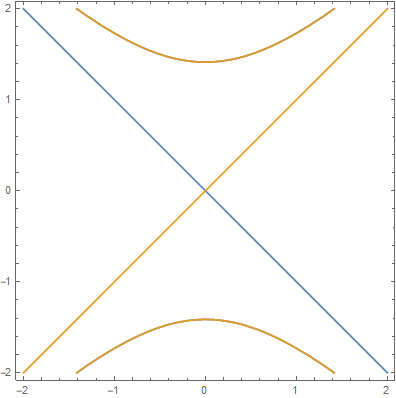}
    \caption{The zero loci for $\pone$ (blue) and $\ptwo$ (orange)}
    \label{figure helicoid}
\end{figure}

\subsection{Absence of isometric deformation in $\R^3$ from generalised catenoids to generalised helicoids}

In this subsection we provide further descriptions of the $C^{1,1}$ isometrically immersed generalised helicoids found by Cao--Huang--Wang in \cite{chw1, chw2}. As is well-known in classical differential geometry ({\it e.g.}, Abbena--Salamon--Gray \cite{asg}), the standard helicoids and standard catenoids, both immersed in $\R^3$, can be {\em isometrically deformed} into each other in $\R^3$ --- there exists a one-parameter family of differentiable immersions $\{\Phi_\tau:\R^2 \map \R^3 :0\leq \tau \leq \pi/2\}$ that varies smoothly in $\tau$, in which $\Phi_0$ is an isometric immersion for the standard helicoid, $\Phi_{\pi/2}$ is that for the standard catenoid, and the induced metrics by $\Phi_\tau$ are constant in $\tau$. Thus, if we know that $\Phi_{\pi/2}$ is an isometric immersion for the standard catenoid, and the family $\{\Phi_\tau\}$ is an isometric deformation, then we can immediately deduce that the standard helicoid is isometrically immersible  in $\R^3$.

Now, recall that in Chen--Slemrod--Wang \cite{cswb}, Cao--Huang--Wang \cite{chw1} and Christoforou \cite{christofourou} (the ranges of parameters in these works differ slightly) it has been proved that the metrics of the  ``generalised catenoid-type'' can be $C^{1,1}$ isometrically immersed in $\R^3$. As a subset of such ``generalised catenoid-type'' metrics, the following metrics
\begin{equation}\label{g cat beta}
g^\beta_{\rm cat}= \begin{bmatrix}
E(y) = \Big(c \cosh\big(\frac{y}{c}\big)\Big)^{\frac{2}{\beta^2-1}} &0\\
0& \frac{1}{c^2(\beta^2-1)^2}E(y)
\end{bmatrix}, \qquad c \neq 0, \,\, \beta \geq \sqrt{2}
\end{equation}
admit isometric immersions $f_{\rm cat}$ on $\R \times [-y_0, 0]$ for some $y_0 > 0$:
\begin{align}\label{f cat beta}
&f^\beta_{\rm cat} (x,y) \nonumber\\
=\,& \bigg(\Big(c\cosh\big(\frac{y}{c}\big)\Big)^{\frac{1}{\beta^2-1}}\sin x,  \Big(c\cosh\big(\frac{y}{c}\big)\Big)^{\frac{1}{\beta^2-1}}\cos x, \int^y \frac{1}{\beta^2-1}\Big(c\cosh\big(\frac{s}{c}\big)\Big)^{\frac{1}{\beta^2-1}-1} \,\dd s \bigg)^\top.
\end{align}
Notice that $\beta = \sqrt{2}$ corresponds to the standard catenoid; in the sequel, the metric and the standard parametrisation are denoted as $g_{\rm std\,\, cat}$ and $f_{\rm std\,\, cat}$, respectively:
\begin{equation}\label{g std cat beta}
g_{\rm std \,\, cat}=g^{\sqrt{2}}_{\rm cat}=\begin{bmatrix}
c^2\cosh\big(\frac{y}{c}\big)^2 &0\\
0& \frac{1}{c^2}\cosh\big(\frac{y}{c}\big)^2 
\end{bmatrix},
\end{equation}
\begin{equation}\label{f std cat beta}
f_{\rm std \,\, cat}(x,y) = f^{\sqrt{2}}_{\rm cat}(x,y) = \bigg(c\cosh\big(\frac{y}{c}\big)\sin x, c\cosh\big(\frac{y}{c}\big)\cos x ,y\bigg)^\top.
\end{equation}

On the other hand, as motivated by the case of classical helicoid and catenoid, a natural parametrisation of the ``generalised helicoids'' is as follows:
\begin{equation}\label{f a psi helicoid}
f^{\alpha, \psi}_{\rm hel}(x,y)=\bigg(\Big(c\sinh\big(\frac{y}{c}\big)\Big)^\alpha\sin x,-\Big(c\sinh\big(\frac{y}{c}\big)\Big)^\alpha \cos x, \psi(x) \bigg)^\top,
\end{equation}
where $\psi$ is a function in $x$ only, and $\alpha >0$ is a constant. Let us call $g^{\alpha, {\rm id}}_{\rm hel}$ the metrics of the ``generalised helicoid-type'' induced by $f^{\alpha, \psi}_{\rm hel}$, which form a subset of the metrics considered in \cite{chw2}:
\begin{equation}\label{g a psi helicoid}
g^{\alpha, \psi}_{\rm hel} := \Big\{f^{\alpha, \psi}_{\rm hel}\Big\}^*\geu.
\end{equation}
Here $\geu$ is the Euclidean metric on $\R^3$, and the superscript ${}^\ast$ denotes the pullback operator. The special case $\alpha=1, \psi=c\,{\rm id}$ gives the parametrisation of the standard helicoid (modulo the transform $y \mapsto c\sinh (yc^{-1})$):
\begin{equation}\label{f std psi helicoid}
f_{\rm std \,\, hel} (x,y) = f^{1, c\,{\rm id}}_{\rm hel} (x,y) = \bigg(c\sinh\big(\frac{y}{c}\big)\sin x,-c\sinh\big(\frac{y}{c}\big) \cos x, cx \bigg)^\top.
\end{equation}
Let us also denote by $g_{\rm std \,\, hel}$ the metric of the standard helicoid, namely
\begin{equation}
g_{\rm std \,\, hel} := \big\{f_{\rm std \,\, hel} \big\}^*\geu.
\end{equation}
Moreover, the family $\{\Phi_\tau: 0 \leq \tau \leq \frac{\pi}{2}\}$ given by
\begin{equation}
\Phi_\tau(x,y):= \sin \tau\cdot f_{\rm std \,\, hel}(x,y) + \cos \tau\cdot f_{\rm std\,\,cat}(x,y)
\end{equation}
is the desired family of isometric deformations. 

In what follows, we show that the generalised catenoids $g^\beta_{\rm cat}$ cannot be naturally isometrically deformed into $g^{\alpha, \psi}_{\rm hel}$, unless both $g^\beta_{\rm cat}=g_{\rm std\,\,cat}$ and $g^{\alpha, \psi}_{\rm hel}=g_{\rm std\,\,hel}$:

\begin{proposition}\label{propn: no deformation}
Let $\{\Phi_\tau: \R^2 \map \R^3: 0 \leq \tau \leq \pi/2\}$ be the natural one-parameter family of parametrisations 
\begin{equation*}
\Phi_\tau := \sin \tau \cdot f^{\alpha, \psi}_{\rm hel} + \cos \tau \cdot f^\beta_{\rm cat},
\end{equation*}
where $f^{\alpha, \psi}_{\rm hel}$ and $f^\beta_{\rm cat}$ are defined as in Eqs.\,\eqref{f a psi helicoid} and \eqref{f cat beta}. Then $\Phi_\tau$ is an isometric deformation from the generalised catenoid to the generalised helicoid if and only if
\begin{equation*}
\alpha =1,\qquad \psi =c\,{\rm Id}\qquad \text{ and }\beta = \sqrt{2}.
\end{equation*}
\end{proposition}

\begin{proof}
	Let $\gtau$ be the metric induced by $\Phi_{\tau}$, namely $\gtau:=(\Phi_{\tau})^\ast \geu$. Then
\begin{align*}
\gtau_{12} &= \frac{\p \Phi_\tau}{\p x} \cdot \frac{\p \Phi_\tau}{\p y}\\
&= \begin{bmatrix}
\cos(\tau) \Big[c\sinh \big(\frac{y}{c}\big)\Big]^\alpha \cos x - \sin(\tau) \Big[c\cosh\big(\frac{y}{c}\big)\Big]^{\frac{1}{\beta^2-1}}\sin x\\
\cos(\tau)\Big[c\sinh \big(\frac{y}{c}\big)\Big]^\alpha \sin x + \sin(\tau ) \Big[c\cosh\big(\frac{y}{c}\big)\Big]^{\frac{1}{\beta^2-1}}\cos x\\
\cos(\tau)\psi'(x)
\end{bmatrix} \cdot \\
&\quad \begin{bmatrix}
\cos(\tau)\alpha  \Big[c\sinh \big(\frac{y}{c}\big)\Big]^{\alpha-1}\cosh\big(\frac{y}{c}\big) \sin x + \sin \tau \frac{1}{\beta^2-1}\Big[c\cosh\big(\frac{y}{c}\big)\Big]^{\frac{1}{\beta^2-1}-1} \sinh\big(\frac{y}{c}\big)\cos x\\
-\cos(\tau)\alpha  \Big[c\sinh \big(\frac{y}{c}\big)\Big]^{\alpha-1}\cosh\big(\frac{y}{c}\big) \cos x + \sin \tau \frac{1}{\beta^2-1}\Big[c\cosh\big(\frac{y}{c}\big)\Big]^{\frac{1}{\beta^2-1}-1} \sinh\big(\frac{y}{c}\big)\sin x\\
\sin(\tau) \frac{1}{\beta^2-1}\Big[c\cosh\big(\frac{y}{c}\big)\Big]^{\frac{1}{\beta^2-1}-1}
\end{bmatrix}\\
&= \cos(\tau)\sin(\tau) c^{\alpha-1+\frac{1}{\beta^2-1}}\Big[\cosh\big(\frac{y}{c}\big)\Big]^{\frac{1}{\beta^2-1}-1}\Big[\sinh \big(\frac{y}{c}\big)\Big]^{\alpha-1} \bigg\{-\alpha\Big[\cosh\big(\frac{y}{c}\big)\Big]^2 +\\
&\quad  \frac{1}{\beta^2-1}\Big[\sinh \big(\frac{y}{c}\big)\Big]^2\bigg\}+ \cos(\tau)\sin(\tau) \frac{1}{\beta^2-1} \psi'(x)\Big[\cosh\big(\frac{y}{c}\big)\Big]^{\frac{1}{\beta^2-1}-1}. 
\end{align*}	

Thus, for $\Phi_\tau$ to be independent of $\tau$, we need
\begin{equation}\label{h}
0 = \frac{\psi'(x)}{\beta^2-1} +  c^{\alpha-1+\frac{1}{\beta^2-1}}\Big[\sinh \big(\frac{y}{c}\big)\Big]^{\alpha-1} \bigg\{-\alpha\Big[\cosh\big(\frac{y}{c}\big)\Big]^2 +\frac{1}{\beta^2-1}\Big[\sinh \big(\frac{y}{c}\big)\Big]^2\bigg\}. 
\end{equation}
The first term on the right-hand side of Eq.\,\eqref{h} depends only on $x$, and the second term only on $y$. So both terms must be constants, and, in particular, $\psi'(x)=0$. Now let us introduce $z:=\sinh(y/c)$; in this new variable, the second term (call it $S$) can be expressed as
\begin{equation*}
S(z) = c^{\alpha-1+\frac{1}{\beta^2-1}} z^{\alpha-1}\Big\{-\alpha +\big(\frac{1}{\beta^2-1}-\alpha\big)z^2\Big\}.
\end{equation*}
This forces $\alpha=1$ and $\frac{1}{\beta^2-1} - \alpha =0$, namely $\beta = \sqrt{2}$. Therefore, $S(z)=-c$, and by Eq.\,\eqref{h} we can deduce $\psi(x)=cx$. The necessity has thus been established.

\smallskip

For sufficiency, we compute 
\begin{align*}
	\gtau_{11} &= \frac{\p \Phi_\tau}{\p x} \cdot \frac{\p \Phi_\tau}{\p x}\\
	&= \bigg\{\cos \tau \Big(c\sinh\big(\frac{y}{c}\big)\Big)^\alpha\cos x + \sin\tau \Big(c\cosh\big(\frac{y}{c}\big) \Big)^{\frac{1}{\beta^2-1}} \sin x \bigg\}^2 \\
&\qquad\qquad +\bigg\{ -\cos\tau \Big(c\sinh\big(\frac{y}{c}\big)\Big)^\alpha \sin x + \sin\tau \Big(c\cosh\big(\frac{y}{c}\big) \Big)^{\frac{1}{\beta^2-1}} \cos x \bigg\}^2 + \cos^2\tau [\psi'(x)]^2 \\
&= \cos^2\tau \Big(c\sinh\big(\frac{y}{c}\big)\Big)^{2\alpha } + \sin^2\tau \Big(c\cosh\big(\frac{y}{c}\big) \Big)^{\frac{2}{\beta^2-1}} + \cos^2\tau [\psi'(x)]^2,
	\end{align*}
as well as
\begin{align*}
\gtau_{22} &= \frac{\p \Phi_\tau}{\p y} \cdot \frac{\p \Phi_\tau}{\p y}\\
	&= \bigg\{\cos\tau\cdot \alpha \Big[c\sinh\big(\frac{y}{c}\big)\Big]^{\alpha-1}\cosh \big(\frac{y}{c}\big) \sin x+ \sin\tau \frac{1}{\beta^2-1} \Big[c\cosh\big(\frac{y}{c}\big)\Big]^{\frac{1}{\beta^2-1}-1}\sinh\big(\frac{y}{c}\big) \cos x\bigg\}^2 \\
	&\quad + \bigg\{-\cos\tau \cdot \alpha \Big[c\sinh\big(\frac{y}{c}\big)\Big]^{\alpha-1} \cosh \big(\frac{y}{c}\big)  \cos x + \sin\tau \frac{1}{\beta^2-1} \Big[c\cosh\big(\frac{y}{c}\big)\Big]^{\frac{1}{\beta^2-1}-1}\sinh\big(\frac{y}{c}\big)\sin x \bigg\}^2 \\
	&\quad + \sin^2\tau \frac{1}{(\beta^2-1)^2} \Big(c^2\cosh^2\big(\frac{y}{c}\big)\Big)^{\frac{1}{\beta^2-1}-1}\\
	&= \cos^2\tau \,\alpha^2 \Big[c\sinh\big(\frac{y}{c}\big)\Big]^{2\alpha -2} \cosh^2\big(\frac{y}{c}\big) + \sin^2\tau \frac{1}{(\beta^2-1)^2} \Big(\cosh\big(\frac{y}{c}\big)\Big)^{\frac{2}{\beta^2-1}} c^{\frac{2}{\beta^2-1}-2}.
\end{align*}
In the case $\alpha=1, \beta=\sqrt{2}$ and $\psi(x)=cx$, the above identities together with $1+\sinh^2\theta=\cosh^2\theta$ immediately gives us
\begin{equation*}
\gtau_{11}=\gtau_{22}=c^2\cosh^2\big(\frac{y}{c}\big), \qquad \gtau_{12} = 0 \qquad \text{for all } 0\leq \tau\leq\frac{\pi}{2}.
\end{equation*}
Indeed, $\Phi_\tau$ is independent of $\tau$. The proof is now complete.  \end{proof}

Proposition \ref{propn: no deformation} suggests that the ``generalised helicoids'' considered in Cao--Huang--Wang \cite{chw1, chw2} and the ``generalised catenoids'' considered in Chen--Slemrod--Wang \cite{cswb}, Cao--Huang--Wang \cite{chw1} and Christoforou \cite{christofourou} are distinct geometric objects, unless they are precisely the classical helicoids and catenoids. In particular, one cannot prove the existence of isometric immersions of the latter via a natural isometric deformation in the ambient space $\R^3$ from the former. This provides further characterisation of the families of isometrically immersible metrics found in \cite{chw1, chw2, cswb, cswc, christofourou}.

\section{Generalised Enneper metrics}
In this section, we consider the following one-parameter family of metrics:
\begin{equation}\label{g, enneper}
\gaa = \begin{bmatrix}
(1+x^2 + y^2)^\alpha & 0  \\
0& (1+x^2 + y^2)^\alpha
\end{bmatrix}.
\end{equation}
When $\alpha = 2$, the metric $g^{(2)}$ gives the well-known classical minimal surface, known as the ``Enneper surface'' (see \cite{asg}). It has the following global isometric immersion:
\begin{equation}
f_{\rm enn}(x,y):=\bigg(x-\frac{1}{3}x^3 + xy^2, \, -y-x^2y + \frac{1}{3}y^3, \, x^2-y^2 \bigg)^\top,
\end{equation}
with the Gauss curvature
\begin{equation}
\kappa_{\rm enn} = \frac{-4}{(1+x^2+y^2)^4}.
\end{equation}

We call the metrics $\{\gaa\}$ in Eq.\,\eqref{g, enneper} ``generalised Enneper metrics''. For simplicity in notations, introduce
\begin{equation}
\TT(x,y):= 1+x^2+y^2,
\end{equation}
so that $E=G=\TT^\alpha$, $F = 0$ in the metric $\gaa$. Thus, we have
\begin{equation*}
\begin{cases}
E_x =2\alpha \TT^{\alpha -1}x, \quad E_y = 2\alpha \TT^{\alpha-1}y,\\
E_{xx} = 2\alpha\TT^{\alpha-1} + 4\alpha(\alpha-1)\TT^{\alpha-2}x^2,\quad E_{yy} = 2\alpha\TT^{\alpha-1} + 4\alpha(\alpha-1)\TT^{\alpha-2}y^2.
\end{cases}
\end{equation*}
The Christoffel symbols can be computed from Eq.\,\eqref{Gamma}:
\begin{align*}
&\G^1_{11} = \frac{\alpha x}{\TT}, \qquad \G^1_{12} = \frac{\alpha y}{\TT}, \qquad \G^1_{22} = -\frac{\alpha x}{\TT},\\
&\G^2_{11} = -\frac{\alpha y}{\TT}, \qquad \G^2_{12} = \frac{\alpha x}{\TT}, \qquad \G^2_{22} = \frac{\alpha x}{\TT}.
\end{align*}
Next, the non-trivial component of Riemann curvature becomes
\begin{align*}
R^1_{212} = \p_1 \G^1_{22} - \p_2\G^1_{12} + \G^1_{22}\G^1_{11} + \G^2_{22}\G^1_{21} - (\G^1_{12})^2 - \G^2_{12}\G^1_{22} = \frac{2\alpha}{\TT^2};
\end{align*}
thus 
\begin{equation}\label{kappa for enneper}
\kappa = \frac{g_{11}R^1_{212}}{|g|} = -2\alpha \TT^{-2-\alpha} < 0 \qquad \text{ whenever } \alpha >0.
\end{equation}
In addition, we have
\begin{equation*}
\gamma = \sqrt{-\kappa} = \frac{\sqrt{2\alpha}}{\TT^{2+\alpha}} \,\, \text{ and } \,\, \na \gamma = -\big(2+{\alpha}\big) \sqrt{2\alpha} \TT^{-2-\frac{\alpha}{2}}\begin{bmatrix}
x\\
y
\end{bmatrix},
\end{equation*}
so
\begin{equation*}
\frac{\na\gamma}{\gamma} = -\frac{2+\alpha}{\TT} \begin{bmatrix}
x\\
y
\end{bmatrix};
\end{equation*}
on the other hand, 
\begin{equation}
\frac{\na E}{E} = \frac{2\alpha \TT^{\alpha-1}\begin{bmatrix}
x\\
y
\end{bmatrix}}{\TT^\alpha} = \frac{2\alpha}{\TT}\begin{bmatrix}
x\\
y
\end{bmatrix}.
\end{equation}
Therefore, we also have
\begin{equation}
\frac{\na\gamma}{\gamma} = \beta\frac{\na E}{E} \qquad \text{ where } \,\,\beta = \frac{-2\alpha}{2+\alpha}.
\end{equation}

The source terms $\sone$ and $\stwo$ are given as follows:
\begin{align}\label{S1, enneper}
\sone &= -\frac{1}{2} \frac{2\alpha y}{\TT} (u+v)^2(u-v) + \frac{\alpha x}{\TT} (u+v)^2 + \beta v(u+v) \frac{2\alpha x}{\TT} \nonumber\\
&\quad + \frac{1}{2}\frac{2\alpha y}{\TT} u + \Big\{-\frac{1}{2}\frac{2\alpha y}{\TT}+ \beta \frac{2\alpha y}{\TT}\Big\} v + \frac{1}{2} \frac{2\alpha x}{\TT}\nonumber\\
&= \frac{\alpha y}{\TT} \bigg\{-(u+v)^2(u-v) + \mu (u+v)^2 - \frac{2\alpha}{2+\alpha} \mu v(u+v) + u - \frac{5\alpha + 2}{2+\alpha} v + \mu \bigg\}
\end{align}
and
\begin{align}\label{S2, enneper}
\stwo &= -\frac{1}{2} \frac{2\alpha y}{\TT} (u+v)(u-v)^2 + \frac{\alpha x}{\TT} u^2 + \Big\{\frac{\alpha x}{\TT} + \beta \frac{2\alpha x}{\TT}\Big\}v^2\nonumber\\
&\quad - \Big\{\frac{2\alpha x}{\TT} + \beta \frac{2\alpha x}{\TT} \Big\} uv + \Big\{\frac{3}{2} \frac{2\alpha y }{\TT} + \beta \frac{2\alpha y}{\TT}\Big\} u - \Big\{\frac{\alpha y}{\TT} + \beta\frac{2\alpha y}{\TT}\Big\}v+ \frac{1}{2}\frac{2\alpha x}{\TT}\nonumber\\
&= \frac{\alpha y}{\TT} \bigg\{-(u+v)(u-v)^2 + \mu u^2 + \frac{2-3\alpha}{2+\alpha} \mu v^2 - \frac{4-2\alpha}{2+\alpha}\mu uv - \frac{2-3\alpha}{2+\alpha}(u+v) +\mu  \bigg\}.
\end{align}
Throughout the current section we write $\mu$ for
\begin{equation}
x \equiv \mu y.
\end{equation}

Thus, for $y\geq 0$, the signs of $\sone,\stwo$ are equal to those of the polynomials $\pone, \ptwo$:
\begin{eqnarray}
&&\pone:=-(u+v)^2(u-v) + \mu (u+v)^2 - \frac{2\alpha}{2+\alpha} \mu v(u+v) + u - \frac{5\alpha + 2}{2+\alpha} v + \mu,\,\,\,\\
&&\ptwo := -(u+v)(u-v)^2 + \mu u^2 + \frac{2-3\alpha}{2+\alpha} \mu v^2 - \frac{4-2\alpha}{2+\alpha}\mu uv - \frac{2-3\alpha}{2+\alpha}(u+v) +\mu.\,\,\,
\end{eqnarray}
The zero loci of the above polynomials in the $(u,v)$-plane are fairly complicated. Fig.\,\ref{figure enneper} shows $\{\pone=0\}$ and $\{\ptwo=0\}$ for a typical choice of $\alpha,\mu$, and in the complimentary material in \S 8 we include the animation of the two-parameter ($\alpha,\mu$) family of the zero loci. Nevertheless, one distinctive common feature about these zero loci can be observed: the manner in which the rightmost branches of $\{\pone=0\}$ and $\{\ptwo=0\}$ intersect guarantees the existence of a invariant region $Q$ with vertices
$$
\bullet_t = (a+b,b),\qquad \bullet_r = (a+2b, 0),\qquad \bullet_d = (a+b,-b),\qquad \bullet_l = (a,0)
$$
with some $a$ and $b>0$, due to Proposition \ref{propn: inv regions VIP}. Furthermore, by experimenting with different ranges of the parameters $\alpha, \mu$, we find that for each $\alpha\in [1,10]$ and $\mu \in [0,1]$, the above constants $a,b$ can be chosen independent of $\mu$. By arguments similar to those in the previous sections, we can deduce:
\begin{theorem}\label{thm: enneper}
Let $y_0>0$ be arbitrary and denote the initial data by 
$$
(u,v)|_{y=y_0} =: (u_0, v_0) : \R \map \R^2.
$$
For each $\alpha \in [1,10]$, there exist positive numbers $a,b>0$ such that the following holds: Assume that
\begin{equation}
a \leq \inf_\R \,(u_0+v_0) \leq \sup_\R (u_0+v_0) \leq a+2b, \qquad -a \geq \sup_\R \,(u_0 - v_0) \leq \inf_\R (u_0 - v_0) \leq -a-2b.
\end{equation}
Then there exists an $L^\infty$ weak solution to the Gauss--Codazzi equations \eqref{gauss eq} and \eqref{codazzi eqs} in the  interior of the wedge $\Omega = \{(x,y)\in\R^2: x \geq 0, y \geq 0, y\leq x\}$ for the generalised Enneper metric
\begin{equation}
\gaa = \begin{bmatrix}
(1+x^2 + y^2)^\alpha & 0  \\
0& (1+x^2 + y^2)^\alpha
\end{bmatrix}.
\end{equation}
\end{theorem}
By the symmetry of the $x,y$ variables in $\gaa$, Theorem \ref{thm: enneper} implies the existence of $C^{1,1}$ isometric immersions in the whole $\R^2$, with possible exception on the diagonals $\{y=\pm x\}$.

\begin{figure}[h]
    \centering
\includegraphics[scale=0.5]{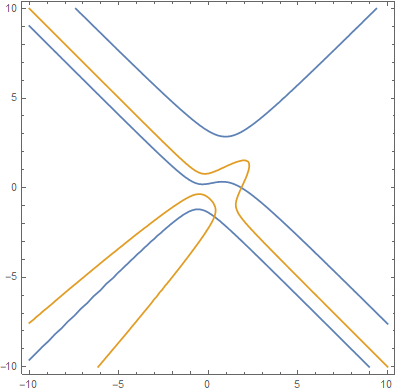}
    \caption{The zero loci for $\pone$ (blue) and $\ptwo$ (orange), with $c=1, d=2$}
    \label{figure enneper}
\end{figure}

\section{``Reciprocal-type'' metrics: Non-existence of invariant regions}
In this section we discuss a special class of negatively curved metrics, for which the source terms $\sone, \stwo$ are easy to compute, but the invariant regions $Q$ of the form in Proposition \ref{propn: inv regions VIP} fail to exist. It manifests the potential limitations of the compensated compactness/invariant region method. 

We consider the metrics of the ``reciprocal-type'':
\begin{equation}\label{g, reciprocal type}
g=\begin{bmatrix}
E(y) & 0 \\
0& E(y)^{-1}
\end{bmatrix}, \qquad E(y)>0.
\end{equation}
The Brioschi's formula (see \cite{asg}) gives us
\begin{equation}\label{kappa for reciprocal type}
\kappa = -\frac{1}{2\sqrt{EG}}\bigg\{\frac{\p}{\p x} \Big(\frac{G_x}{\sqrt{EG}}\Big)+\frac{\p}{\p y}\Big(\frac{E_y}{\sqrt{EG}}\Big)\bigg\} = -\frac{E_{yy}}{2}.
\end{equation}
Hence, we require $E$ to be strictly convex and positive. Thanks to Eqs.\,\eqref{Gamma}\eqref{tilde Gamma} we have
\begin{equation*}
\tg^2_{22}=-\frac{E_y}{2E}+\frac{\gamma_y}{\gamma}, \quad \tg^2_{11} = -\frac{EE_y}{2}, \quad \tg^1_{12} = \frac{E_y}{2E} + \frac{\gamma_y}{2\gamma}, \quad \text{other } \tg^i_{jk} = 0.
\end{equation*}
The source terms in Eqs.\,\eqref{S1}\eqref{S2} thus become:
\begin{align}
&\sone = -\frac{3E_y}{2E} u + \big(-\frac{E_y}{2E}+\frac{\gamma_y}{\gamma}\big)v -\frac{E_y}{2E}(u-v)(u+v)^2,\nonumber\\
&\stwo = -\frac{3E_y}{2E} u+ \big(\frac{E_y}{2E}-\frac{\gamma_y}{\gamma}\big)v -\frac{E_y}{2E}(u-v)^2(u+v).
\end{align}

To proceed, let us find conditions on $E$ such that $\gamma_y/\gamma$ is proportional to $E_y/E$: in this case we can factor out $E_y/E$ in $\sone, \stwo$. Indeed, suppose there exists a constant $\alpha$ (to be determined) such that
\begin{equation}
\frac{\gamma_y}{\gamma} = \alpha \frac{E_y}{E}.
\end{equation}
By Lemma \ref{lemma: derivatives of kappa and gamma} and Eq.\,\eqref{kappa for reciprocal type}, we obtain the ODE:
\begin{equation}\label{Eyy}
\frac{E_{yyy}}{E_{yy}} = 2\alpha \frac{E_y}{E}.
\end{equation}
Taking logarithmic on both sides, there is a constant $C>0$ such that
\begin{equation}
E_{yy} = CE^\alpha.
\end{equation}
From the above ODE, we find that exponential and hyperbolic trigonometric functions are good choices for $E=E(y)$. 

\subsection{Case 1: exponential}
In this case let us consider
\begin{equation}
E(y) = Ae^{\omega y}.
\end{equation}
Thus, $C=\omega>0$ and $\alpha =1$ in Eq.\,\eqref{Eyy}. We can easily compute that 
\begin{equation}
\kappa = -\frac{\omega^2e^{\omega y}}{2} < 0,
\end{equation}
as well as
\begin{equation*}
\frac{E_y}{E} = \frac{2\gamma_y}{\gamma} = \omega.
\end{equation*}
Moreover, the source terms are given by
\begin{align}
&\sone = -\frac{3\omega}{2} u - \frac{\omega e^{2\omega y}}{2} (u-v)(u+v)^2,\nonumber\\
&\stwo = -\frac{3\omega}{2} u - \frac{\omega e^{2\omega y}}{2} (u-v)^2(u+v),
\end{align} 
whose signs are determined by $\pone,\ptwo \in \R[u,v]$ respectively:
\begin{eqnarray}
&&\pone = 3u + k(u-v)(u+v)^2,\\
&&\ptwo = 3u + k(u+v)(u-v)^2,
\end{eqnarray}
with the factor
\begin{equation}
k\equiv k(\omega,y) = e^{2\omega y} > 0.
\end{equation}

\subsection{Case 2: hyperbolic cosine}
In this case we consider
\begin{equation}
E(y) = A \cosh(\omega y), \qquad A, \omega>0.
\end{equation}
Then by Eq.\,\eqref{Eyy} we have $\alpha = 1/2$, and 
\begin{equation}
\kappa = -\omega^2\cosh(\omega y)/2<0.
\end{equation}
Also, we have
\begin{equation*}
\frac{E_y}{E} = \frac{2\gamma_y}{\gamma} = \omega \tanh (\omega y),
\end{equation*}
and the source terms are given by
\begin{align}
&\sone=-\frac{\omega\tanh(\omega y)}{2}\Big\{3u + [\cosh(\omega y)]^2 (u-v)(u+v)^2\Big\},\nonumber\\
&\stwo = -\frac{\omega\tanh(\omega y)}{2}\Big\{3u+ [\cosh(\omega y)]^2 (u-v)^2(u+v)\Big\}.
\end{align}
Thus, whenever $y>0$, the signs of $\sone,\stwo$ are opposite to those of
\begin{eqnarray}
&&\pone = 3u + \widetilde{k}(u-v)(u+v)^2,\\
&&\ptwo = 3u + \widetilde{k}(u+v)(u-v)^2,
\end{eqnarray}
with the factor
\begin{equation}
\widetilde{k} \equiv \widetilde{k}(\omega,y) = \cosh^2(\omega y)\geq 1.
\end{equation}
Here $\pone,\ptwo$ take the same form as those in Case 1, \S 6.1; the only difference is in the factor $\widetilde{k}$.

As seen from Fig.\,\ref{figure k} ({\it cf.} the supplementary materials in \S 9 for the animation of the one-parameter family $k \in [0,3]$), the zero locus of $\pone$ ($\ptwo$, resp.) in the $(u,v)$-plane is a curve locally look like $\{v=u^3\}$ ($\{v=u^3\}$, resp.) near the origin. Moreover, in the right half the $(u,v)$-plane to the zero locus of $\pone$ ($\ptwo$, resp.) one has $\pone > 0$ ($\ptwo >0$, resp.) It is clear that the invariant region $Q$ of the form in Proposition \ref{propn: inv regions VIP} cannot exist. Therefore, our method in \S \S 4--6 does not apply to the reciprocal-type metrics \eqref{g, reciprocal type}.

\begin{figure}[h]
    \centering
\includegraphics[scale=0.5]{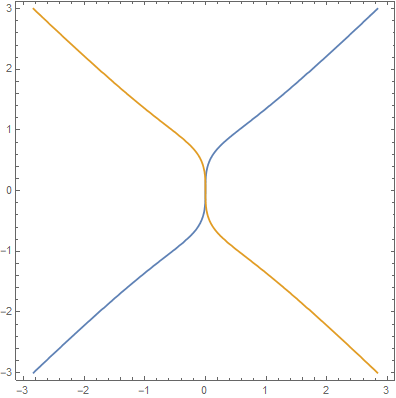}
    \caption{The zero loci for $\pone$ (blue) and $\ptwo$ (orange), with $k=1.52$}
    \label{figure k}
\end{figure}

\section{Conclusions}

In this paper we study isometric immersions of negatively curved surfaces by the method of compensated compactness. The existence of $C^{1,1}$ isometric immersions is guaranteed by the existence of invariant regions $Q$; the latter is deduced from the relative positions of the zero loci of source terms $\sone, \stwo$ of the hyperbolic balance laws associated to the Gauss--Codazzi system, namely Eqs.\,\eqref{wz 1} and \eqref{wz 2}. 

In Theorem \ref{thm: H2} we show the existence of isometric immersions of the standard Lobachevsky plane away from the $x$-axis into $\R^3$. In Theorems \ref{thm: helicoid} and \ref{thm: helicoid 2} we prove the existence of isometric immersions of an infinite family of helicoid-type metrics into $\R^3$. Finally, in Theorem \ref{thm: enneper} we prove the existence of isometric immersions, with computer assistance, for a one-parameter family of metrics containing the Enneper surface. The regularity for the above isometric immersions is $C^{1,1}$, in view of the method of compensated compactness we adopt in this paper. 

The supplementary materials to this work can be downloaded from Wolfram Cloud, \texttt{Siran  Li$\_$isometric immerions of negatively curved surfaces.\,Mathematica code.nb.}

\bigskip
\noindent
{\bf Acknowledgement}.
Siran Li would like to thank Professors Gui-Qiang Chen, Pengfei Guan, Dmitry Jakobson and Marshall Slemrod for many insightful discussions. Moreover, he is deeply indebted to Dr.\,Xiaochun Meng, Miss Yikai Chen and Miss Yishu Gong for their generous help on graphing software. Part of this work has been done during the author's stay as a CRM--ISM postdoctoral fellow at the Centre de Recherches Math\'{e}matiques, Universit\'{e} de Montr\'{e}al and the Institut des Sciences Math\'{e}matiques. Siran Li would like to thank these institutes for their hospitality.

\end{document}